\documentclass[12pt]{amsart}

\begin{document}



\newtheorem{theorem}{Theorem}[section]
\newtheorem{prop}[theorem]{Proposition}
\newtheorem{lemma}[theorem]{Lemma}
\newtheorem{definition}[theorem]{Definition}
\newtheorem{cor}[theorem]{Corollary}
\newtheorem{example}[theorem]{Example}
\newtheorem{remark}[theorem]{Remark}
\newtheorem{assumption}[theorem]{Assumption}
\newcommand{\ra}{\rightarrow}
\renewcommand{\theequation}
{\thesection.\arabic{equation}}
\newcommand{\ccc}{{\mathcal C}}
\newcommand{\one}{1\hspace{-4.5pt}1}

\def\HSL { H^1_{L,S}(X) }

\def \Gg {\widetilde{{\mathcal g}}_{L}}
\def \GG {{\mathcal g}_{L}}
\def \SL {\sqrt[m]{L}}
\def \sq {\sqrt}
\def \GL {{\mathcal G}_{\lambda,L}^{\ast}}

\def \l {\lambda}
\def \gL{{\widetilde {\mathcal G}}_{\lambda, L}^{\ast}}
\def \RN {\mathbb{R}^n}
\def\RR{\mathbb R}
\newcommand{\nf}{\infty}
 \def \Lips  {{   \Lambda}_{L}^{ \alpha,  s }(X)}
\def\BL {{\rm BMO}_{L}(X)}
\def\HAL { \mathbb{F}\dot{\mathbb{H}}_{L,at,M}^{\lambda}({\mathbb R}^n)}
\def\HML { \mathbb{F}\dot{\mathbb{H}}_{L,mol,M,\epsilon}^{\lambda}({\mathbb R}^n) }
\def\HM{ H^p_{L, {mol}, 1}(X) }
\def\Ma { {\mathcal M} }
\def\MM { {\mathcal M}^{ 2, \lambda, \epsilon}(L) }
\def\dMM { ({\mathcal M}^{2, \lambda,\epsilon}(L))^{\ast} }

\def\HSL { H^p_{L, S_h}(X) }
\newcommand\mcS{\mathcal{S}}
\newcommand\mcB{\mathcal{B}}
\newcommand\D{\mathcal{D}}
\newcommand\C{\mathbb{C}}
\newcommand\N{\mathbb{N}}
\newcommand\R{\mathbb{R}}
\newcommand\G{\mathbb{G}}
\newcommand\T{\mathbb{T}}
\newcommand\Z{\mathbb{Z}}
\allowdisplaybreaks

 \title[Preduals of quadratic Campanato spaces with heat kernel bounds]
 {Preduals of quadratic Campanato spaces associated to operators with heat kernel bounds}

\author[Liang Song, Jie Xiao \ and \ Xuefang Yan]{Liang Song,   \  Jie Xiao\ and\   Xuefang Yan}
\thanks{{\it {\rm 2010} Mathematics Subject Classification:}
 42B35,   47B38.}

\address{Liang Song, Department of Mathematics, Sun Yat-sen
University, Guangzhou, 510275, P.R. China} \email{songl@mail.sysu.edu.cn}

\address{Jie Xiao, Department of Mathematics and Statistics, Memorial University, St. John's NL,
A1C 5S7, Canada} \email{jxiao@mun.ca}

\address{Xuefang Yan, Department of Mathematics, Sun Yat-sen
University, Guangzhou, 510275, P.R. China \& College of mathematics and information
science, Heibei Normal University, Shijiazhuang, 050016, P.R. China}
 \email{yanxuefang2008@163.com; yanxuefang@mail.hebtu.edu.cn}

 \date{}

\bigskip

\begin{abstract}
Let $L$ be a nonnegative, self-adjoint operator on $L^2(\mathbb{R}^n)$ with the Gaussian upper bound on its heat kernel.
As a generalization of the square Campanato space $\mathcal{L}^{2,\lambda}_{-\Delta}(\mathbb R^n)$, in \cite{DXY} the quadratic Campanato space $\mathcal{L}_L^{2,\lambda}(\mathbb{R}^n)$ is defined by a variant of the maximal function associated with the semigroup $\{e^{-tL}\}_{t\geq 0}$. On the basis of \cite{DX} and \cite{YY} this paper addresses the preduality of $\mathcal{L}_L^{2,\lambda}(\mathbb{R}^n)$ through an induced atom (or molecular) decomposition. Even in the case $L=-\Delta$ the discovered predual result is new and natural.
\end{abstract}

\maketitle

\tableofcontents

\section{Introduction}

Given $1\leq p<\infty$ and $0<\lambda<n$. A locally integrable complex-valued function $f$ on $\mathbb{R}^n$ is said to belong to be in the Morrey space $L^{p,\lambda}(\mathbb{R}^n)$ provided
\begin{align*}
\big\|f\big\|_{L^{p,\lambda}}=\sup\limits_{B\subset \mathbb{R}^n}\left(r_B^{-\lambda}\int_B |f(x)|^p \ dx\right)^{1/p}<\infty,
\end{align*}
where $r_B$ is the radius of the ball $B$. Such a function space was introduced by C. B. Morrey in \cite{Mo} to treat the solutions of some quasi-linear elliptic PDEs. Since then, the theory of Morrey spaces has been developed extensively; see e.g. \cite{AX1, AX2, BRV, P, K, Ta} and the references therein.

To weaken the integral condition appeared in the Morrey space, in his 1963/4 papers \cite{Cam, Ca} S. Campanato utilized the modified mean oscillation to define the following function space:
$$
f\in \mathcal{L}^{p,\lambda}(\mathbb{R}^n)\Longleftrightarrow
\|f\|_{\mathcal{L}^{p,\lambda}}=\sup\limits_{B\subset \mathbb{R}^n}\left(r_B^{-\lambda}\int_B |f(x)-f_B|^p \ dx\right)^{1/p}<\infty,
$$
where $f_B:=|B|^{-1}\int_B f(y)\ dy$. It is easy to see that $L^{p,\lambda}(\mathbb{R}^n)$ is a proper subclass of $\mathcal{L}^{p,\lambda}(\mathbb{R}^n)$ since any complex constant is in $\mathcal L^{p,\lambda}(\mathbb{R}^n)\setminus L^{p,\lambda}(\mathbb{R}^n)$.
Interestingly, $\mathcal{L}^{p,\lambda}(\mathbb{R}^n)$ under $(p,\lambda)\in (1,\infty)\times (0,n)$ exists as a dual space - more precisely, if $Z^{q,\lambda}(\mathbb{R}^n)$ with $q=p/(p-1)$ stands for the Zorko space (cf.\cite{Z}) of all functions $f$
on $\mathbb{R}^n$ with the norm
$$
\big\|f\big\|_{Z^{q,\lambda}}=\inf\Big\{\|\{c_k\}\|_{l^1}: \ f=\sum_k c_ka_k\Big\}<+\infty,
$$
where $a_k$ is a $(q, \lambda)$-atom and $\|\{c_k\}\|_{l^1}<+\infty$, and the infimum is taken
over all possible functions $f=\sum_k c_ka_k$ whose $a_k$ is a $(q,\lambda)$-atom on $\mathbb R^n$:
\begin{itemize}
\item $a_k$ is supported on
a ball $B\subseteq \mathbb{R}^n$;

\item $\int a(x)\,dx=0$;

 \item $\|a\|_q\leq r_B^{-\lambda/p}\ \ \hbox{with}\ \ 1/q+1/{p}=1,$

\end{itemize}
then
$$
(Z^{q,\lambda}(\mathbb{R}^n))^*=\mathcal{L}^{p,\lambda}(\mathbb{R}^n),
$$
namely, the Zorko $Z^{q,\lambda}(\mathbb{R}^n)$ is identified with a predual of the Campanato space $\mathcal{L}^{p,\lambda}(\mathbb{R}^n)$.

 However, there are important situations in which the standard theory of function spaces is not applicable, including certain problems in the theory of partial differential operators generalizing the Laplacian. There is a need to consider the function spaces that are adapted to a linear operator $L$, similarly to the way that such function spaces as the above-defined Campanato spaces are adapted to the Laplacian. This topic has attracted a lot of attention, and has been a very active research topic in harmonic analysis, potential theory and PDEs; see, for instance, \cite{DZ,ADM,DY1,DY2,AMR,HM,HMM,HLMMY,DL,JY,SY}.

For our purpose, we will consider such a nonnegative self-adjoint operator $L$ on $L^2({{\mathbb R}^n})$
that the semigroup $e^{-tL}$, generated by $-L$ on $L^2({{\mathbb R}^n})$, has the kernel $p_t(x,y)$ obeying the Gaussian upper bound for a constant $C>0$:
\begin{equation}\label{e2.1}
|p_t(x,y)|\leq C  t^{-n/2}\exp\Big(-{|x-y|^2\over ct}\Big)\ \ \  {\rm for\ all} \ t>0\  \& \ {\rm for\ a.e.}\ (x,y)\in \mathbb R^n\times{{\mathbb R}^n}.
\end{equation}
Such an upper bound condition is a typical one needed in the theory of elliptic or sub-elliptic differential operators of second order; see, for example, \cite{Da}.

Keeping in mind that the quadratic Campanato space $\mathcal{L}^{2,\lambda}(\mathbb R^n)$ is a prime example in the family of the Campanato spaces that are useful in analysis and PDEs (see e.g. \cite{X, Xr} and their references), and following \cite{DXY}, we say that a function $f$ belongs to
the space $\mathcal{L}_L^{2,\lambda}(\mathbb{R}^n)$ provided
\begin{align*}
\|f\|_{\mathcal{L}_L^{2,\lambda}}=\sup\limits_{B\subset \mathbb{R}^n} \left(r_B^{-\lambda}\int_B |f(x)-e^{-r^2_BL}f(x)|^2 \ dx\right)^{1/2}<\infty,
\end{align*}
 where $r_B$ is the radius of the ball $B$. Here, the function  $e^{-r_B^2L}f$ is seen as an average version of $f$ (at the scale $r_B^2$) and replaces the mean value $f_B$ in the definition of the Campanato space $\mathcal{L}^{2,\lambda}(\mathbb R^n)$. For this idea and its applications, we refer the reader to \cite{DM,Ma,DY1,DY2,DXY}, and especially point out that if $L$ equals the nonnegative Laplace operator $-\Delta=-\sum_{j=1}^n\partial^2/\partial x_j^2$ on $\mathbb{R}^n$, then
 $\mathcal{L}_L^{2,\lambda}(\mathbb{R}^n)$ coincides with $\mathcal{L}^{2,\lambda}(\mathbb{R}^n)$, see \cite[Proposition 8]{DXY}. Hence,
 $\mathcal{L}_L^{2,\lambda}(\mathbb{R}^n)$ generalizes  $\mathcal{L}^{2,\lambda}(\mathbb{R}^n)$.

Needless to say, the study of $\mathcal{L}_L^{2,\lambda}(\mathbb{R}^n)$ is far beyond completeness. Nevertheless, being inspired by the Choquet integral against the Hausdorff capacity used in \cite{DX} and \cite{YY}, in this paper we can at least obtain the following description of the preduality for $\mathcal{L}_L^{2,\lambda}(\mathbb{R}^n)$ as one of the fundamental problems of the Campanto function theory associated to a nonnegative self-adjoint operator with the Gaussian kernel bound \eqref{e2.1}.

 \begin{theorem}\label{t11} Let $L$ be a nonnegative self-adjoint operator obeying \eqref{e2.1} and  $\Lambda^{(\infty)}_\lambda$ be $(0,n)\ni\lambda$-dimensional Hausdorff capacity.  If $F\dot{H}_{L}^{\lambda}({\mathbb R}^{n})$ stands for the completion of
$$\left\{f\in L^2({\mathbb R}^n): \big\|f\big\|_{F\dot{H}_{L}^{\lambda}}
=\inf_{\omega}\left(\int_{{\mathbb R}^{n+1}_{+}}{|t^2Le^{-t^2L}f(x)|^{2}\omega(x,t)^{-1}} \frac{dxdt}{t}\right)^{1/2}<\infty\right\}
$$
in the norm $\|\cdot\|_{F\dot{H}_{L}^{\lambda}}$, where the infimum is taken over all nonnegative Borel functions $\omega$ on ${\mathbb R}^{n+1}_{+}$ with its non-tangential maximal function $\mathsf{N}\omega$ satisfying the Choquel integral condition $\int_{\mathbb R^n}\mathsf{N}\omega\,d\Lambda^{(\infty)}_\lambda\le 1$, then the predual of the space $\mathcal{L}_L^{2,\lambda}({\mathbb R}^n)$
is ${F}\dot{H}_{L}^{\lambda}({\mathbb R}^n)$, namely,
$$
\big({F}\dot{H}_{L}^{\lambda}({\mathbb R}^n)\big)^*=\mathcal{L}_L^{2,\lambda}({\mathbb R}^n).
$$
\end{theorem}

The proof of the above preduality theorem proceeds via the forthcoming three sections. In Section 2, we recall some basic facts about Choquet integrals and square tent spaces. In Section 3, we make the space $F\dot{H}^{\lambda}_L(\mathbb{R}^n)$ more transparent via giving its atomic (or molecular) decomposition. In Section 4, we use a new description of $\mathcal{L}_L^{2,\lambda}({\mathbb R}^n)$ and the atomic
(or molecular) characterization of  $F\dot{H}^{\lambda}_L(\mathbb{R}^n)$ as a tool to complete the argument for Theorem \ref{t11}.

\begin{remark} {\rm(i)} Hopefully, the investigation of $ \mathcal{L}_L^{2,\lambda}({\mathbb R}^n)$ and hence $F\dot{H}^{\lambda}_L(\mathbb{R}^n)$ can be  moved appropriately to a more general setting of operators on metric spaces as described in \cite[Chapter 7]{Ou}.

{\rm (ii)} From now on, the letters $C$ and $c$  will denote (possibly different) constants that are independent of the essential variables.
\end{remark}

\section{Necessary groundwork}

\setcounter{equation}{0}

\subsection{Choquet integrals} We shall work exclusively with the upper half-space  $
 {\Bbb R}^{n+1}_+$.
If $x\in {\Bbb R}^n$, $\Gamma(x)$
will denote the   cone  $\Gamma(x)=\{(y,t)\in {\Bbb R}^{n+1}_+:
|x-y|<t\}$.   For any  set $E\subset{\mathbb R}^n$, the tent over
$E$, $T(E)$, is the set $\{(y,t)\in {\mathbb R}^{n+1}_{+}: B(y,t)\subset E\}$.
The nontangential maximal function $\mathsf{N}f$ of a measurable function $f$ on ${\mathbb R}^{n+1}_{+}$
is defined by $$\mathsf{N}f(x)=\sup_{(y,t)\in\Gamma(x)}|f(y,t)|.
$$

Let us recall the notion of Hausdorff capacities; see, for example,   \cite{A1,A2}.
\begin{definition} If $\lambda\in(0,n)$ and $E\subset \mathbb{R}^n$, then $\lambda$-dimensional Hausdorff capacity of $E$ is defined by
\begin{align*}
\Lambda_\lambda^{(\infty)}(E):=\inf\Big\{\sum_j r_j^\lambda:\ E\subset \bigcup_{j=1}^\infty B(x_j,r_j)\Big\},
\end{align*}
where the infimum is taken over all covers of $E$ by balls $B(x_j,r_j)$ with centers $x_j$ and radii $r_j$.
\end{definition}

A dyadic version of the Hausdorff capacity, $\widetilde{\Lambda}_{\lambda}^{(\infty)}$, was introduced in \cite{YY2}, which is defined by
\begin{align*}
\widetilde{\Lambda}_{\lambda}^{(\infty)}(E)=\inf \big\{\sum_j l(I_j)^\lambda:\ E\subset \big(\bigcup_j I_j\big)^{\rm o} \big\},
\end{align*}
where the infimum ranges only over covers of $E$ by dyadic cubes $\{I_j\}_j$, and $A^{\rm o}$ denotes the interior of the set $A$.

 It is well known that $\lambda$-dimensional Hausdorff capacity ${\Lambda}_{\lambda}^{(\infty)}$ and $\widetilde{\Lambda}_{\lambda}^{(\infty)}$
 are equivalent -- more precisely, there exist positive constants $C_{1}(n,\lambda)$ and $C_2(n,\lambda)$, depending on $n$ and $\lambda$, such that
\begin{eqnarray}\label{e1.3}
C_{1}(n,\lambda)\Lambda_{\lambda}^{(\infty)}(E)\leq\widetilde{\Lambda}_{\lambda}^{(\infty)}(E)\leq C_{2}(n,\lambda)\Lambda_{\lambda}^{(\infty)}(E),
\ \ {\rm for \ all \ } E\subset \mathbb{R}^n.
\end{eqnarray}

Next, we recall a notion of the Choquet integrals with respect to the Hausdorff capacities
 (cf. \cite{A1,A2}): for a function $f:{\mathbb R}^n\rightarrow[0,\infty]$, define
\begin{eqnarray*}
\int_{{\mathbb R}^n}fd\Lambda_{\lambda}^{(\infty)}:=\int_{0}^{\infty}\Lambda_{\lambda}^{(\infty)}(\{x\in{\mathbb R}^n:f(x)>t\})\,dt.
\end{eqnarray*}

\subsection{Square tent spaces} Definitions \ref{def 1.2} \& \ref{def 1.4} below are inspired by   \cite{DX}.

\begin{definition} \label{def 1.2}  Let $\lambda\in(0,n)$. The space  $F\dot{T}^\lambda(\mathbb{R}^{n+1}_+)$
 consists of all Lebesgue measurable functions $f$ on ${\mathbb R}^{n+1}_{+}$
  for which $$\|f\|_{F\dot{T}^\lambda}=
  \inf_{\omega}\bigg(\int_{{\mathbb R}^{n+1}_{+}}\frac{|f(x,t)|^{2}}{\omega(x,t)}\frac{dxdt}{t}\bigg)^{1/2}<\infty,
$$
 where the infimum is taken over all nonnegative Borel functions $\omega$ on ${\mathbb R}^{n+1}_{+}$ with
\begin{eqnarray}\label{e1.2}
\int_{{\mathbb R}^n}\mathsf{N}\omega d\Lambda_{\lambda}^{(\infty)}\leq1,
\end{eqnarray}
and with the restriction that $\omega$ is allowed to vanish only where $f$ vanishes.
\end{definition}

Note that if a function $\omega$ satisfies (\ref{e1.2}), then $\omega(x,t)\leq Ct^{-\lambda}$.
This shows that condition  $\|f\|_{F\dot{T}^\lambda}=0$ implies $f=0$ almost everywhere (see \cite{DX, YY}).

The following lemma shows that  $\|f\|_{F\dot{T}^\lambda}$ satisfies the triangle inequality
with a constant, and then  $\|\cdot\|_{F\dot{T}^\lambda}$ is   a quasi-norm. It can be shown that
the space $F\dot{T}^\lambda({\mathbb R}^{n+1}_+)$ is complete under this quasi-norm.

\begin{lemma}\label{le1.6}   Let $\lambda\in(0,n)$.
If $\sum_{j}\|g_{j}\|_{F\dot{T}^\lambda}<\infty$, then $g=\sum_{j}g_{j}\in F\dot{T}^\lambda({\mathbb R}^{n+1}_+)$ with
\begin{eqnarray*}
\|g\|_{F\dot{T}^\lambda}\leq \sqrt{C_{1}(n,\lambda)^{-1}C_{2}(n,\lambda)}\sum_{j}\|g_{j}\|_{F\dot{T}^\lambda},
\label{e1.5}
\end{eqnarray*}
where $C_{1}(n,\lambda)$ and $C_{2}(n,\lambda)$ are the constants in (\ref{e1.3}).
\end{lemma}

\begin{proof} The proof follows from a slight modification of an argument as in   \cite[Lemma 5.3]{DX}. We omit the detail here.
\end{proof}

\begin{definition} \label{def 1.4} Let $\lambda\in(0,n)$. A function $a$ on ${\mathbb R}^{n+1}_{+}$ is
said to be an $F\dot{T}^\lambda$-atom associated with a ball $B$, if $a$ is supported in $T(B)$ and satisfies
\begin{eqnarray*}
\int_{T(B)}|a(x,t)|^{2}\frac{dxdt}{t}\leq{r_{B}^{-\lambda}}.
\end{eqnarray*}
\end{definition}

\medskip

\noindent Recall that the space $T^2_2({\mathbb R}^{n+1}_+)$ is a tent space (see \cite{CMS}), which is defined by

\begin{eqnarray*}
T^2_2({\mathbb R}^{n+1}_+)=\Big\{ f(y,t):\ x\mapsto \Big(\int_{0}^{\infty}\!\!\int_{|y-x|<t}|f(y,t)|^{2}
\frac{dydt}{t^{n+1}}\Big)^{1/2}\ \hbox{is\ in}\ L^{2}({\mathbb R}^n)\Big\}.
\end{eqnarray*}

 \noindent
\begin{theorem}\label {th1.7}  Let $\lambda\in(0,n)$. Then the following results hold:

{\rm (i)} $f\in F\dot{T}^\lambda({\mathbb R}^{n+1}_+)$
if and only if there is a sequence $\{a_j\}$ of $F\dot{T}^\lambda$-atoms and an $l^{1}$-sequence $\{\lambda_{j}\}$ such that
\begin{align}\label{ee3.1}
f=\sum_{j}\lambda_{j}a_{j}.
\end{align}
 Moreover,
$$
\|f\|_{F\dot{T}^\lambda}\approx \inf \{\sum_{j}|\lambda_{j}|:f=\sum_{j}\lambda_{j}a_{j}\},
$$
where the infimum is taken over all possible forms $f$ in (\ref{ee3.1}). The right hand side thus
defines a norm on $F\dot{T}^\lambda({\mathbb R}^{n+1}_+)$ which makes it into a Banach spaces.

{\rm (ii)} If $f\in F\dot{T}^\lambda({\mathbb R}^{n+1}_+)\cap T_2^2({\mathbb R}^{n+1}_+)$, then the
 decomposition (\ref{ee3.1}) also converges in $T^2_2({\mathbb R}^{n+1}_+).$
\end{theorem}

\begin{proof}
The proof of (i) is similar to that of  \cite[Theorem 4.1]{YY} or \cite[Theorem 5.4]{DX}.
For the proof of (ii), we can follow an  argument of  \cite[Proposition 4.10]{HLMMY}) to show it, and so omit details here.
\end{proof}

\section{The space $F\dot{H}_{L}^{\lambda}(\mathbb R^n)$}
\setcounter{equation}{0}

\subsection{An atomic decomposition of $F\dot{H}_{L}^{\lambda}({\mathbb R}^{n})$} Given a nonnegative self-adjoint operator $L$ on $L^2({\mathbb R^n})$ satisfying \eqref{e2.1}. For any $(x,t)\in {\mathbb R}^{n}\times (0, +\infty)={\mathbb R}^{n+1}_+$
and for every $f\in L^2({\mathbb R}^n)$, define
$$
\begin{cases}
P_tf(x)=e^{-tL}f(x)=\int_{{\mathbb R}^n} p_t(x,y)f(y)dy;\\
Q_tf(x)=t Le^{-tL}f(x)=\int_{{\mathbb R}^n} -t\Big(\frac{d p_t(x,y)}{dt}\Big)f(y)dy.
\end{cases}
$$
Then, like $p_t(x,y)$ obeying \eqref{e2.1} the kernel $q_{t}(x,y)$ of
$Q_{t}$ satisfies
\begin{equation}\label{e2.2}
|q_{t}(x,y)|\leq C t^{-n/2}\exp\Big(-{|x-y|^2\over ct}\Big) \ \  {\rm for\ all} \ t>0\  \& \ {\rm for\ a.e.}\ (x,y)\in \mathbb R^n\times{{\mathbb R}^n}.
\end{equation}
See, for instance, \cite[Theorem~6.17]{Ou}.

\begin{definition} \label{def 2.1}  Let $\lambda\in(0,n)$ and $L$ be a nonnegative self-adjoint operator obeying \eqref{e2.1}. The space $F\dot{H}_{L}^{\lambda}({\mathbb R}^{n})$
is defined to be the completion of
$$\left\{f\in L^2({\mathbb R}^n): \big\|f\big\|_{F\dot{H}_{L}^{\lambda}}
=\|t^2Le^{-t^2L}(f)\|_{F\dot{T}^\lambda}<\infty\right\}
$$
in the norm $\|\cdot\|_{F\dot{H}_{L}^{\lambda}}$.
\end{definition}

\begin{definition} \label{def 2.2} Given $(M,\lambda)\in{\mathbb N}\times(0, n)$ and $L$, a nonnegative self-adjoint operator enjoying \eqref{e2.1}.

{\rm (i)} A function $a\in L^2({\mathbb R}^n)$ is called a
$(2,M,\lambda)$-atom associated to the operator $L$ if there exist a function
$b\in {\mathcal D}(L^M)$ and a ball $B$ such that

\begin{itemize}

\item $a=L^M b$;

\item {\rm supp}\  $L^{k}b\subset B, \ k=0, 1, \dots, M$;

\item $\|(r_B^2L)^{k}b\|_{L^2({\mathbb R}^n)}\leq r_B^{2M-\lambda/2},\ k=0,1,\dots,M$.
\end{itemize}

{\rm (ii)} We say that $\sum\lambda_j a_j$ is an atomic
$(2,M,\lambda)$-representation of f if $ \{\lambda_j\}_{j=0}^{\infty}\in {\ell}^1$,
each $a_j$ is a $(2,M,\lambda)$-atom, and the sum converges in $L^2({\mathbb R}^n).$  Set
\begin{equation*}
\HAL\,:=\Big\{f: \mbox{f has an atomic $(2,M,\lambda)$-representation} \Big\},
\end{equation*}
with the norm given by
\begin{align*}
&\|f\|_{\HAL}\\
&={\inf}\Big\{\sum_{j=0}^{\infty}|\lambda_j|:
f=\sum\limits_{j=0}^{\infty}\lambda_ja_j
\ \mbox{is an atomic $(2,M,\lambda)$-representation}\Big\}.
\end{align*}
{\rm (iii)} The space $F\dot{H}_{L,at,M}^{\lambda}({\mathbb R}^{n})$ is then defined as the completion of  $\HAL$ with respect to $\|\cdot\|_{\HAL}$.
\end{definition}

Recall that, if $E_L(\lambda)$ denotes the spectral
decomposition of a nonnegative self-adjoint operator $L$ on $L^2(\mathbb R^n)$, then for every bounded Borel function
$F:[0,\infty)\to{\Bbb C}$, one defines the operator
$F(L): L^2(\RN)\to L^2(\RN)$ by the formula

\begin{eqnarray}\label{mL}
F(L):=\int_0^{\infty}F(\lambda)dE_L(\lambda).
\end{eqnarray}

Hence the operator $\cos(t\sqrt{L})$ is
well-defined on $L^2(\RN)$ for all $t>0$. Thus it makes sense to make the following definition.

\begin{definition}\label{finitespeed} A nonnegative self-adjoint operator
$L$ is said to satisfy the finite speed propagation property for solutions
of the corresponding wave equation if there exists a constant $c_0>0$ such that

\begin{eqnarray}\label{es3.8}
\langle \cos(t\sqrt{L})f_1, \,  f_2\rangle=\int \cos(t\sqrt{L})f_1(x)\overline{f_2(x)}\,dx=0
\end{eqnarray}

\noindent for all $0<c_0 t<d(U_1, U_2)$ and $U_i\subset \RN$, $f_i\in L^2(U_i)$,
$i=1,2$.
\end{definition}

In particular, if $K_{\cos(t\sqrt{L})}(x,y)$ denotes the integral kernel
of the operator $\cos(t\sqrt{L})$, then (\ref{es3.8}) entails that for every
$t>0$,

\begin{eqnarray}\label{es3.9}
{\rm supp}\ K_{\cos(t\sqrt{L})}\subseteq {\mathcal D}_t:=
\Big\{(x,y)\in \RN\times \RN: \,|x-y|\leq c_0t\Big\}.
\end{eqnarray}

\begin{prop}\label{prop3.3} Let $L$ be a nonnegative self-adjoint operator
acting on $L^2(\RN)$. Then (\ref{e2.1}) implies (\ref{es3.8}).
\end{prop}

\begin{proof} The argument follows from \cite[Theorem~2]{Si2}
and \cite[Theorem 3.4]{CS}.
\end{proof}

From Proposition~\ref{prop3.3} and (\ref{e2.1}) it follows that the kernel
$K_{\cos(t\sqrt{L})}(x,y)$ of the operator ${\cos(t\sqrt{L})}$ has the
property (\ref{es3.9}). By the Fourier inversion formula, whenever $F$ is
an even bounded Borel function with $\hat{F} \in L^1(\mathbb{R})$,
we can  write $F(\sqrt{L})$ in terms of
$\cos(t\sqrt{L})$. Concretely, by recalling (\ref{mL}) we have
$$
F(\sqrt{L})=(2\pi)^{-1}\int_{-\infty}^{\infty}{\hat F}(t)\cos(t\sqrt{L})\,dt,
$$
which, when combined with (\ref{es3.9}), gives

\begin{eqnarray*}
K_{F(\sqrt{L})}(x,y)=(2\pi)^{-1}\int_{|t|\geq c_0^{-1}d(x,y)}{\hat F}(t)
K_{\cos(t\sqrt{L})}(x,y)\,dt.
\end{eqnarray*}

\begin{lemma}\label{lemma2.4}  Assume $L$  is a nonnegative self-adjoint operator  on $L^2({\mathbb R^n})$ satisfying \eqref{e2.1}.  Let $\varphi\in C^{\infty}_0(\mathbb R)$ be
even, $\mbox{supp}\,\varphi \subset [-c_0^{-1}, c_0^{-1}]$, where $c_0$ is
the constant in (\ref{es3.9}). Let
$\Phi$ denote the Fourier transform of $\varphi$. Then for each
$\kappa=0,1,\dots$, and for every $t>0$, the kernel
$K_{(t^2L)^{\kappa}\Phi(t\sqrt{L})}(x,y)$ of
$(t^2L)^{\kappa}\Phi(t\sqrt{L})$ satisfies

\begin{eqnarray*}
\label{e2.4}\hspace{0.5cm} {\rm supp}\ \!
K_{(t^2L)^{\kappa}\Phi(t\sqrt{L})}  \subseteq \big\{(x,y)\in
{\Bbb R}^n\times {\Bbb R}^n: |x-y|\leq t\big\}.
\end{eqnarray*}
\end{lemma}
\begin{proof} This follows from \cite[Lemma 3.5]{HLMMY}.
\end{proof}

In what follows, let $\varphi$, $c_0$,
and $\Phi$ be as in Lemma~\ref{lemma2.4}, but with an extra assumption that
$$
\varphi \geq 0\ \ \&\ \ \varphi \geq c > 0\ \ \hbox{on}\ \ (-1/(2c_0), 1/(2c_0)).
$$
For $M\in \mathbb{N}$ set
$$
\Psi(x):=x^{2(M+1)}\Phi(x), \quad\forall \ x\in{\mathbb{R}}.
$$
Consider the operator
$\pi_{\Psi,L}: T^2_2({\mathbb R}^{n+1}_+)\rightarrow L^2({\mathbb R}^n)$, given by
\begin{eqnarray*}
\pi_{\Psi,L}(F)(x):= \int_0^{\infty}\Psi(t\sqrt L)\big(F(\cdot,\, t)\big)(x){dt\over t},
\end{eqnarray*}
where the improper integral converges weakly in $L^2$. The bound
\begin{equation}\label{e2.12}
\|\pi_{\Psi,L}F\|_{L^2} \leq C\|F\|_{T^2_2}
\end{equation}
 follows readily by duality and the $L^2$ quadratic estimate.
Moreover, we have the following analogue of the well-known argument of \cite[Theorem 6]{CMS}.

\begin{lemma}\label{le2.5} Given a nonnegative self-adjoint operator $L$ obeying \eqref{e2.1}. Suppose $A$ is an $F\dot{T}^\lambda$-atom associated to a ball
$B$.  Then there is a  constant $C$, depending only on $\Psi$,  such that $C^{-1}\, \pi_{\Psi,L} (A)$
is a $(2,M,\lambda)$-atom associated to  $2B$.
\end{lemma}

\begin{proof} Fix a ball $B$ and let $A$ be an $F\dot{T}^\lambda$-atom  associated to $B$.  Thus,
$$\int_{T(B)}|A(x,t)|^2\frac{dxdt}{t}\leq |B|^{-\lambda/n}.$$
For $M\in\mathbb N$ write
$$
a:=\pi_{\Psi,L}(A)=L^M b,
$$
where
$$
b:=\int_0^{\infty}t^{2M}t^2L\Phi(-t\sqrt{L})\big(A(\cdot,\, t)\big){dt\over t}.
$$
Observe that the functions $L^k b,\, k=0,1,..., M,$ are supported on the ball $2B$, by
Lemma~\ref{lemma2.4}, since $A$ is supported in $T(B)$.
Consider
some $g\in L^2(2B)$ such that $\|g\|_{L^2(2B)}=1$.
Then for every $k=0,1,\dots,M $ we have

\begin{align*}
&\Big|\int_{{\mathbb R}^n}(r_B^2L)^{k}b(x)\,g(x)dx\Big|\nonumber\\
&=\Big|\lim_{\delta\to 0}\int_{{\mathbb R}^n}\left(\int_\delta^{1/\delta}t^{2M}r_B^{2k}  L^{k}t^2L \Phi(t\sqrt{L})
\big(A(\cdot,\, t)\big)(x)
{dt\over t}\right)g(x)\,dx \Big|\nonumber \\
&=\Big|\int_{T(B)} A(x,t)  t^{2M}r_B^{2k} L^{k}t^2L\Phi(t\sqrt{L})g(x)
{dxdt\over t} \Big|\\
&\leq  r_B^{2M}\Big(
\int_{ T(B)  }\big|  A(x,t)
\big|^2{dxdt\over t}\Big)^{1/2} \Big(
\int_{ T(B)  }\big|  (t^2L)^{k+1}\Phi(t\sqrt{L})g(x)
\big|^2{dxdt\over t}\Big)^{1/2} \nonumber \\
&\leq Cr_B^{2M}|B|^{-\lambda/2n} \|g\|_{L^2(2B)},\nonumber
\end{align*}
where the fact that
$A$ is an $F\dot{T}^\lambda$-atom supported in $T(B)$ (hence, $0<t<r_B$) has been used, and  the last inequality follows from the $L^2$ quadratic estimate:
$$
\Big(
\int_0^{\infty}\int_{\mathbb R^n} \big|  (t^2L)^{k+1}\Phi(t\sqrt{L})g(x)
\big|^2{dxdt\over t}\Big)^{1/2}  \\[4pt]
 \leq C \|g\|_{L^2({\mathbb R^n})}.
$$
As a consequence, one gets
\begin{eqnarray*}
\|(r_B^2L)^{k}b\|_{L^2(2B)}
&\leq& Cr_B^{2M}|B|^{-\lambda/2n},\quad k=0,1,...,M.
\end{eqnarray*}
The proof is complete.
\end{proof}

\begin{theorem}\label{th2.7} Assume $L$  is   a nonnegative self-adjoint operator   on $L^2({\mathbb R^n})$ satisfying \eqref{e2.1}. Then
$$
F\dot{H}_{L,at,M}^{\lambda}({\mathbb R}^{n})= F\dot{H}_{L}^{\lambda}({\mathbb R}^{n}).
$$
Moreover
$$\|f\|_{F\dot{H}_{L,at,M}^{\lambda}} \approx \|f\|_{F\dot{H}_{L}^{\lambda}},$$
where the implicit constants depend only on the pair $(M,\lambda)\in\mathbb N\times(0,n)$ and the constant in \eqref{e2.1}.
\end{theorem}
\begin{proof} On the one hand, we show
$$
\HAL \subseteq (L^2({\mathbb R}^n)\cap F\dot{H}_{L}^{\lambda}({\mathbb R}^{n})).
$$

Note that $\HAL \subseteq L^2({\mathbb R}^n)$. Indeed, by definition, a $(2,M,\lambda)$-atom belongs to
$R(L)$, and therefore so does any finite linear combination of atoms. Moreover,  every $f\in\HAL$ is an $L^2$ limit of such a finite linear combination. Meanwhile, we are required to verify $\HAL  \subseteq F\dot{H}_{L}^{\lambda}$. To do so, let
$f \in \HAL$, where $f = \sum \lambda_j a_j$ is an atomic $(2,M,\lambda)$-representation such that
$$
\|f\|_{\mathbb{F}\dot{\mathbb{H}}_{L,at}^{\lambda}({\mathbb R}^n)} \approx  \sum_{j=0}^\infty |\lambda_j| .
$$
Since the sum converges in $L^2$ (by Definition \ref{def 2.2}), and an application of the $L^{2}$ boundedness of $Q_{t^{2}}$, we have that
$$
|Q_{t^{2}}(f)| \leq \sum_{j=0}^\infty |\lambda_j| \,|Q_{t^{2}}(a_j)|.
$$
By Lemma~\ref{le1.6}, it will be enough to show that for
every $(2,M,\lambda)$-atom $a$ associated to a ball $B=B(x_{B},r_{B})$, we have
$\|Q_{t^{2}}(a)\|_{F\dot{T}^\lambda}\leq C$. Now, fix $\delta=(n-\lambda)/2>0$ and let
\begin{eqnarray}\label{e1.11}
\widetilde{\omega}(x,t)=kr_{B}^{-\lambda}\min\bigg\{1,\Big(\frac{r_{B}}{\sqrt{|x-x_{B}|^{2}+t^{2}}}\Big)^{\lambda+\delta}\bigg\},
\end{eqnarray}
where $k$ will be chosen below.  Since for $x\in {\mathbb R}^n$, the distance in ${\mathbb R}^{n+1}_{+}$
 from the $\Gamma(x)$ to $(x_{B},0)$ is $|x-x_{B}|/\sqrt{2}$, the nontangential maximal function
  of $\widetilde{\omega}$ is bounded by $$N\widetilde{\omega}(x)\leq
  kr_{B}^{-\lambda}\min\bigg\{1,\Big(\frac{\sqrt{2}r_{B}}{|x-x_{B}|}\Big)^{\lambda+\delta}\bigg\}$$
and so
\begin{eqnarray*}
 k^{-1}\int_{{\mathbb R}^n}N\widetilde{\omega} d\Lambda_{\lambda}^{(\infty)}
&\leq&\int_{{\mathbb R}^n}r_{B}^{-\lambda}\min\{1,(\frac{\sqrt{2}r_{B}}
{|x-x_{B}|})^{\lambda+\delta}\}\ d\Lambda_{\lambda}^{(\infty)}\nonumber\\
&=&\int_{0}^{\infty}\Lambda_{\lambda}^{(\infty)}\{x\in{\mathbb R}^n:r_{B}^{-\lambda}
\min\{1,(\frac{\sqrt{2}r_{B}}{|x-x_{B}|})^{\lambda+\delta}\}>\alpha\}\ d\alpha\nonumber\\
&\leq&\int_{0}^{r_{B}^{-\lambda}}\Lambda_{\lambda}^{(\infty)}\{B(x_{B},
(\alpha r_{B}^{\lambda})^{\frac{-1}{\lambda+\delta}}\sqrt{2}r_{B})\})\ d\alpha=C.
\end{eqnarray*}
Upon choosing $k=C^{-1}$ to make $\widetilde{\omega}$ satisfy (\ref{e1.2}), we see
\begin{eqnarray*}
\|Q_{t^{2}}(a)\|_{F\dot{T}^\lambda}^{2}&\leq &\int_{{\mathbb R}^{n+1}_{+}}
 \frac{|Q_{t^{2}}a(x)|^{2}}{\widetilde{\omega}(x,t)}\frac{dxdt}{t}\nonumber\\
&\leq&\int_{T(2B)} \frac{|Q_{t^{2}}a(x)|^{2}}{\widetilde{\omega}(x,t)}
\frac{dxdt}{t}+\sum_{j=2}^{\infty}\int_{T(2^{j}B)\setminus T(2^{j-1}B)} \frac{|Q_{t^{2}}a(x)|^{2}}{\widetilde{\omega}(x,t)}\frac{dxdt}{t}\nonumber\\
&=:&A_{0}+\sum_{j=2}^{\infty}A_{j}.
\end{eqnarray*}
Note that
$$
\widetilde{\omega}(x,t)\geq k2^{-(\lambda+\delta)} r_{B}^{-\lambda}\ \ \hbox{on}\ \
T(2B).
$$
So, by using (\ref{e2.2}) and the definition of $(2,M,\lambda)$-atom we obtain
\begin{eqnarray*}
A_{0}\leq Cr_{B}^{\lambda}\int_{{\mathbb R}^{n+1}_{+}} |Q_{t^{2}}a(x)|^{2}\ \frac{dxdt}{t}\leq Cr_{B}^{\lambda}\|a\|_{L^{2}}^{2}\leq C.
\end{eqnarray*}
\noindent
To estimate $A_{j}$ for $j=2,3, \cdots$, notice that for each $(x,t)\in T(2^{j}B)\setminus T(2^{j-1}B)$ one has
$$
\widetilde{\omega}(x,t)=Cr_{B}^{-\lambda}\Big(\frac{r_{B}}{\sqrt{|x-x_{B}|^{2}+t^{2}}}\Big)^{\lambda+\delta}
\geq Cr_{B}^{-\lambda}2^{-j(\lambda+\delta)}.
$$
Thus
\begin{eqnarray}\label{e4.1}
A_{j}\leq Cr_{B}^{\lambda}2^{j(\lambda+\delta)}\int_{T(2^{j}B)\setminus T(2^{j-1}B)} |Q_{t^{2}}a(x)|^{2}\ \frac{dxdt}{t}.
\end{eqnarray}
\noindent
 For $(x,t)\in T(2^{j}B)\setminus T(2^{j-1}B)$ and $y\in B$, one has  $t+|x-y|\geq 2^{j-1}r_{B}$. By (\ref{e2.2}),  we get
 \begin{align*}
 &\int_{T(2^{j}B)\setminus T(2^{j-1}B)} |Q_{t^{2}}a(x)|^{2}\
 \frac{dxdt}{t}\\
 &\leq C \int_{T(2^{j}B)\setminus T(2^{j-1}B)} \bigg(\int \frac{t \ |a(y)|}{(t+|x-y|)^{n+1}}dy\bigg)^{2}\ \frac{dxdt}{t}\nonumber\\
 &\leq C\frac{1}{(2^{j}r_{B})^{2n+2}}\|a\|_{L^{1}}^{2} \int_{T(2^{j}B)\setminus T(2^{j-1}B)}t \ {dxdt}\nonumber\\
 &\leq C\frac{1}{(2^{j}r_{B})^{2n+2}}\|a\|_{L^{2}}^{2}|B|(2^{j}r_{B})^{2}|2^{j}B|\nonumber\\
&\leq C{2^{-jn}}{r_{B}^{-\lambda}},
 \end{align*}
 which, combining with (\ref{e4.1}), implies
  \begin{eqnarray*}
 A_{j}\leq C2^{-j(n-\lambda-\delta)}=C2^{-j(n-\lambda)/2},
 \end{eqnarray*}
thereby deriving $\|a\|_{F\dot{H}_L^\lambda}\leq C$.

On the other hand, we verify the reverse inequality
$$
F\dot{H}_{L}^{\lambda}({\mathbb R}^{n})\cap L^2({\mathbb R}^n)\subseteq\HAL.
$$
Let
$$
f\in F\dot{H}_{L}^{\lambda}({\mathbb R}^{n})\cap L^2({\mathbb R}^n)\ \ \&\ \
F(\cdot,t):= t^2Le^{-t^2L}f(\cdot).
$$
Note that $F\in F\dot{T}^\lambda\cap T_{2}^{2}$ follows from the definition of $F\dot{H}^\lambda_L$.
So, by Theorem \ref{th1.7} we have
\begin{eqnarray*}
F =\sum_j \lambda_j \, A_j,
\end{eqnarray*}
where each $A_j$ is a $F\dot{T}^\lambda$-atom,
the sum converges in both $T^2_2({\mathbb R}^{n+1}_+)$ and $F\dot{T}^\lambda({\mathbb R}^{n+1}_+)$,
and
\begin{equation}\label{e2.17}\sum\limits_j|\lambda_j|\leq C\|F\|_{F\dot{T}}=C\|f\|_{F\dot{H}_{L}^{\lambda}}.
\end{equation}
Also, by $L^2$-functional
calculus (\cite{Mc}),  we have the ``Calder\'{o}n reproducing formula"
\begin{multline}\label{e2.18}
f(x)=c_{\Psi}\int_0^{\infty}\Psi(t\sqrt L)(t^{2}L e^{-t^2{L}}f)(x) {dt\over t}
=c_\Psi \,\pi_{\Psi,L}(F) = c_\Psi \,\sum_j \lambda_j\, \pi_{\Psi,L}(A_j).
\end{multline}

\noindent where the last sum  converges in $L^2({\mathbb R}^n)$ by (\ref{e2.12}).
Moreover, by Lemma \ref{le2.5} we have that up to multiplication by some
harmless constant $C$, each $a_j := c_\Psi\, \pi_{\Psi,L}(A_j)$
is  a $(2,M,\lambda)$-atom.
Consequently, the last sum in (\ref{e2.18}) is an atomic $(2,M,\lambda)$-representation,
so that $f\in \HAL$, and by (\ref{e2.17}) we have
$$\|f\|_{\HAL}\leq C\|f\|_{F\dot{H}_{L}^{\lambda}},$$
whence deriving the desired inclusion.

The above argument shows that both ${F}\dot{{H}}_{L,at,M}^{\lambda}({\mathbb R}^n)$ and $F\dot{H}_{L}^{\lambda}({\mathbb R}^{n})$ have the same dense subset $\HAL=L^2({\mathbb R}^{n})\cap F\dot{H}_{L}^{\lambda}({\mathbb R}^{n})$ with equivalent norms, and hence they must coincide. This completes the proof.
\end{proof}

\subsection {Relationships between atoms and molecules} In sake of convenience, given a ball $B$ set
\begin{equation}\label{e2.9}\hspace{1cm}
 U_{0}(B)=B, B_{i}=2^{i}B, U_{i}(B)=2^{i}B\setminus2^{i-1}B, \ \ i=1,2,\cdots.
\end{equation}

\begin{definition} \label{def 2.7} Let $\epsilon>0$.  A function $m\in L^2({\mathbb R}^n)$ is called a
$(2,M,\lambda,\epsilon)$-molecule associated to $L$ if there exist a function
$b\in {\mathcal D}(L^M)$ and a ball $B$ such that
\begin{itemize}

\item $m=L^M b$;

\item For every $k=0,1,2,\dots,M$ and $j=0,1,2,\dots$, there holds
$$
\|(r_B^2L)^{k}b\|_{L^2(U_j(B))}\leq r_B^{2M} 2^{-j\epsilon}
(2^jr_{B})^{-\lambda/2},
$$
where the annuli $U_j(B)$ have been defined in (\ref{e2.9}).
\end{itemize}
\end{definition}

\begin{definition}\label{def 2.8} Given $ M\geq 1, \lambda\in (0,n)$ and $\epsilon > 0$.
We say that
$\sum\lambda_j m_j$ is a molecular
$(2,M,\lambda,\epsilon)$-representation of $f$ provided that $ \{\lambda_j\}_{j=0}^{\infty}\in {\ell}^1$,
each $m_j$ is a $(2,M,\lambda,\epsilon)$-molecule, and the sum converges in $L^2({\mathbb R}^{n}).$
Set
\begin{equation*}
\HML= \\\Big\{f:
\mbox{ f has a molecular $(2,M,\lambda,\epsilon)$-representation} \Big\},
\end{equation*}
with the norm given by
\begin{align*}
&\|f\|_{\HML}\\
&={\rm inf}\Big\{\sum_{j=0}^{\infty}|\lambda_j|:
f=\sum\limits_{j=0}^{\infty}\lambda_jm_j\, \mbox{ is a molecular $(2,M,\lambda,\epsilon)$-representation} \Big\}.
\end{align*}
The space $F\dot{H}_{L,mol,M,\epsilon}^{\lambda}({\mathbb R}^{n})$ is then defined as the completion of
 $\mathbb{F}\dot{\mathbb{H}}_{L,mol,M,\epsilon}^{\lambda}({\mathbb R}^n)$ with respect to this norm.
\end{definition}

\begin{lemma}\label{th2.9} Assume $L$  is   a nonnegative self-adjoint operator   on $L^2({\mathbb R^n})$ satisfying \eqref{e2.1}.  Let $\epsilon>0$. Then
$$\HML \subseteq F\dot{H}_{L}^{\lambda}({\mathbb R}^{n})\cap L^2({\mathbb R}^n)$$
and
$$\|f\|_{F\dot{H}_{L}^{\lambda}({\mathbb R}^{n})}\leq C \|f\|_{\HML}.$$
\end{lemma}
\begin{proof} By Lemma \ref{le1.6} and Definition \ref{def 2.1}, and an application of the $L^{2}$
boundedness of $Q_{t^{2}}$, we see that it will be enough to show that for
every $(2,M,\lambda,\epsilon)$-molecule $m$ we have
$\|Q_{t^{2}}(m)\|_{F\dot{T}^\lambda}\leq C$. To this end,
let ${\epsilon}>0$, $m$ be a $(2,M,\lambda,{\epsilon})$-molecule, adapted to the ball $B=B(x_B,r_B)$.
 Fix $\widetilde{\omega}$ as in (\ref{e1.11}) with $\delta<\min((n-\lambda), 2\epsilon)$. Recalling (\ref{e2.9}), we write
\begin{eqnarray*}
\|Q_{t^{2}}(m)\|_{F\dot{T}^\lambda}&\leq& \bigg(\int_{{\mathbb R}^{n+1}_{+}}
\frac{|Q_{t^{2}}(m\chi_{U_{i}(B)})(x)|^{2}}{\widetilde{\omega}(x,t)}\frac{dxdt}{t}\bigg)^{1/2}\nonumber\\
&\leq& \sum_{i=0}^{\infty}\bigg(\int_{{\mathbb R}^{n+1}_{+}} \frac{|Q_{t^{2}}(m\chi_{U_{i}(B)})(x)|^{2}}
{\widetilde{\omega}(x,t)}\frac{dxdt}{t}\bigg)^{1/2}\\
&=:&\sum_{i=0}^{\infty}I_{i}.
\end{eqnarray*}

\noindent For $i=0$, an argument similar to that of Theorem~\ref{th2.7} yields $I_0\leq C$. And, for $i\geq1$ one has
\begin{align*}
I_i&\leq\sum_{j=2}^{\infty}\bigg(\int_{T(2^{j}B_{i})\setminus T(2^{j-1}B_{i})}
 \frac{|Q_{t^{2}}(m\chi_{U_{i}(B)})(x)|^{2}}{\widetilde{\omega}(x,t)}\frac{dxdt}{t}\bigg)^{1/2}\nonumber\\
&+\bigg(\int_{T(2B_{i})} \frac{|Q_{t^{2}}(m\chi_{U_{i}(B)})(x)|^{2}}{\widetilde{\omega}(x,t)}\frac{dxdt}{t}\bigg)^{1/2}\nonumber\\
&=\Big(\sum_{j=2}^{\infty}I_{ij}\Big)+I_{i1}.
\end{align*}

To control $I_{ij} \ (j=2,3,\cdots)$, note that
$$
(x,t)\in T(2^{j+i}B)\setminus T(2^{j+i-1}B)\Longrightarrow\widetilde{\omega}(x,t)\geq Cr_{B}^{-\lambda}2^{-(j+i)(\lambda+\delta)}.
$$
So, a combination of (\ref{e2.2}), the definition of molecules and H\"older's inequality, produces
\begin{align}\label{e2.20}
I_{ij}^{2}&\leq C r_{B}^{\lambda}2^{(i+j)(\lambda+\delta)}
\int_{T(2^{j}B_{i})\setminus T(2^{j-1}B_{i})} {|Q_{t^{2}}(m\chi_{U_{i}(B)})(x)|^{2}}\frac{dxdt}{t}\nonumber\\
&\leq C r_{B}^{\lambda}2^{(i+j)(\lambda+\delta)}\int_{T(2^{j}B_{i})\setminus T(2^{j-1}B_{i})}
\bigg(\int \frac{t|m(y)\chi_{U_{i}(B)}(y)|}{(t+|x-y|)^{n+1}}\,dy\bigg)^{2}\frac{dxdt}{t}\nonumber\\
&\leq C r_{B}^{\lambda}2^{(i+j)(\lambda+\delta)}
\int_{T(2^{j}B_{i})}\frac{t^{2}}{(2^{i+j-2}r_{B})^{2n+2}}\frac{dxdt}{t}\|m\|_{L^{1}(U_{i}(B))}^{2}\\
&\leq C r_{B}^{\lambda}2^{(i+j)(\lambda+\delta)}
\frac{(2^{i+j}r_{B})^{2+n}}{(2^{i+j}r_{B})^{2n+2}}\|m\|_{L^{2}(U_{i}(B))}^{2}|2^{i}B|\nonumber\\
&\leq C2^{-j(n-\lambda-\delta)}2^{-i(2{\epsilon}-\delta)}.\nonumber
\end{align}

To estimate $I_{i1}$,  note that
$$
(x,t)\in T(2B_i)=T(2^{i+1}B)\Longrightarrow\widetilde{\omega}(x,t)\geq Cr_{B}^{-\lambda}2^{-i(\lambda+\delta)}.
$$
Thus, an application of the definition of molecules and (\ref{e2.2}) yields
\begin{align}\label{e2.22}
I^2_{i1}&\leq C r_{B}^{\lambda}2^{i(\lambda+\delta)}\int_{T(B_{i})} |Q_{t^{2}}(m\chi_{U_{i}(B)})(x)|^{2}\frac{dxdt}{t}\nonumber\\
&\leq Cr_{B}^{\lambda}2^{i(\lambda+\delta)}\|m\|^2_{L^{2}(U_{i}(B))}\\
&\leq C2^{-i(2\epsilon-\delta)}.\nonumber
\end{align}

Combining (\ref{e2.20}) and (\ref{e2.22}), we get  $\sum\limits^{\infty}_{i=1}\sum\limits^{\infty}_{j=1} I_{ij}\leq C$, thereby completing the proof.
\end{proof}

As an immediate consequence, we get the following result.

\begin{theorem}\label{th2.10} Suppose $(M,\lambda,\epsilon)\in{\mathbb N}\times(0, n)\times(0,\infty)$ and $L$ is a nonnegative self-adjoint operator obeying \eqref{e2.1}. There holds
$$
F\dot{H}_{L,at,M}^{\lambda}({\mathbb R}^{n})= F\dot{H}_{L,mol,M,\epsilon}^{\lambda}({\mathbb R}^{n}) = F\dot{H}_{L}^{\lambda}({\mathbb R}^{n}).
$$
Moreover
$$
\|f\|_{F\dot{H}_{L,at,M}^{\lambda}} \approx \|f\|_{F\dot{H}_{L}^{\lambda}}\approx \|f\|_{F\dot{H}_{L,mol,M,\epsilon}^{\lambda}},
$$
where the implicit constants depend only on the triple $(M,\lambda,\epsilon)$ and the constant in \eqref{e2.1}.
\end{theorem}

\begin{proof} We have already shown that
$$
F\dot{H}_{L,at,M}^{\lambda}({\mathbb R}^{n}) = F\dot{H}_{L}^{\lambda}({\mathbb R}^{n})\ \ \&\ \
\HAL = F\dot{H}_{L}^{\lambda}({\mathbb R}^{n}) \cap L^2({\mathbb R}^n),
$$
with equivalent norms.  Moreover,
every $(2,M,\lambda)$-atom  is, in particular, a $(2,M,\lambda,\epsilon)$-molecule for
every $\epsilon >0$, hence
$$
\HAL \subseteq \HML\ \
$$
with
$$
\|f\|_{\HML} \leq \|f\|_{\HAL}\quad \forall\ \ f \in \HAL.
$$
Also, by Lemma \ref{th2.9} one has
$$
\HML \subseteq F\dot{H}_{L}^{\lambda}({\mathbb R}^{n}) \cap L^2({\mathbb R}^n)=\HAL
$$
with
$$
\|f\|_{\HAL}\approx\|f\|_{F\dot{H}_{L}^{\lambda}({\mathbb R}^{n})}\leq C \|f\|_{\HML}.
$$
Consequently,
$$
\HML = F\dot{H}_{L}^{\lambda}({\mathbb R}^{n})\cap L^2({\mathbb R}^n)=\HAL,
$$
with equivalent norms. It follows that the three completions
$$
F\dot{H}_{L}^{\lambda}({\mathbb R}^{n});\ \ F\dot{H}_{L,at,M}^{\lambda}({\mathbb R}^{n});\ \
F\dot{H}_{L,mol,M,\epsilon}^{\lambda}({\mathbb R}^{n})
$$
coincide for different choices of $(M,\lambda,\epsilon)\in{\mathbb N}\times(0, n)\times(0,\infty)$.
\end{proof}

\medskip

The following representation will be used later on.

\begin{theorem}\label{th2.11}  Assume $L$  is   a nonnegative self-adjoint operator on $L^2({\mathbb R^n})$ satisfying \eqref{e2.1}. Let $M\ge 1$. Suppose $f=\sum_{i=0}^{N}\lambda_i a_i$,
where $\{a_i\}_{i=0}^N$ is a family of $(2,2M,\lambda)$-atoms and
$\sum_{i=0}^{N}|\lambda_i|<\infty$. Then there is a representation of
$f=\sum_{i=0}^{K}\mu_i m_i$, where the $m_i$'s are
$(2,M,\lambda,M)$-molecules  and
\begin{eqnarray*}
 C_1\|f\|_{\HAL}\leq\sum\limits_{i=0}^{K}|\mu_i|\leq C_2\|f\|_{\HAL},
\end{eqnarray*}
with $C_j=C_j(L,M,\lambda,n)$ for $j=1,2$.
\end{theorem}
\begin{proof}
This follows from a slight modification of the argument for \cite[Theorem ~5.4]{HLMMY}.
\end{proof}

\section{Identification between $\big(F\dot{H}_{L}^{\lambda}(\mathbb R^n)\big)^\ast$ and $\mathcal{L}_{L}^{2,\lambda}(\mathbb R^n)$}
 \setcounter{equation}{0}

\subsection{A characterization of  $\mathcal{L}_{L}^{2,\lambda}(\mathbb R^n)$}
\begin{definition}
\label{de41}
Given a nonnegative self-adjoint operator $L$ on $L^2({\mathbb R^n})$ satisfying \eqref{e2.1}. For $(\nu,\lambda,\epsilon)\in{\mathcal D}(L)\times (0,n)\times(0,\infty)$ let
$$
\phi=L\nu\in L^2({\mathbb R}^n)\ \ \&\ \
\|\phi\|_{\MM}:=\sup_{j\geq 0}\Bigl[2^{j\epsilon}2^{j\frac{\lambda}{2}}\sum_{k=0}^1
\|L^{k} \nu\|_{L^2(U_j(B_0))}\Bigr],
$$
where $B_0$ is the ball centered at some $x_0\in \mathbb{R}^n$ with radius $1$. Then
$$
\MM:=\{\phi=L\nu\in L^2(\mathbb{R}^n): \,\|\phi\|_{\MM}<\infty\}.
$$
\end{definition}

The following two facts are worth mentioning:
\begin{itemize}

\item if $\phi\in\MM$ with norm $1$, then $\phi$ is a
$(2,1,\lambda,\epsilon)$-molecule adapted to $B_0$. Conversely, if $m$ is a
$(2,1,\lambda, \epsilon)$-molecule adapted to any ball, then $m\in\MM$.

\item if $\dMM$ stands for the dual of $\MM$ and $A_t$ denotes either $(I+t^2L)^{-1}$
or $e^{-t^2L}$, then $\dMM\ni f\mapsto (I-A_t)f$ can be determined in the sense of distribution and so this mapping belongs to
$L^2_{\rm loc}({\mathbb R}^n)$ -- indeed, if $\varphi\in L^2(B)$ for some ball $B$,
it follows that
$(I-A_t)\varphi\in\MM$ for every $\epsilon>0$, and so that
$$
\big|\langle (I-A_t)f, \varphi\rangle\big|=
\big|\langle f, (I-A_t)\varphi\rangle\big|\leq C_{t,\, r_B\, {\rm dist}(B, x_0)}
\|f\|_{\dMM}\|\varphi\|_{L^2(B)}.
$$
Similarly, one has $(t^2L)A_tf\in L^2_{\rm loc}({\mathbb R}^n)$.
\end{itemize}

\begin{definition}\label{def 3.1} Given $\lambda\in (0,n)$ and a nonnegative self-adjoint operator $L$ on $L^2({\mathbb R^n})$ satisfying
\eqref{e2.1}. Let
\begin{eqnarray*}
{\mathcal{E}}_\lambda:= \bigcap\limits_{\epsilon>0}\dMM.
\end{eqnarray*}
Then an element $f\in {\mathcal{E}}_\lambda$ is said to belong to $\mathcal{L}_{L}^{2,\lambda}({\mathbb R}^{n})$ provided
\begin{equation*}
\|f\|_{\mathcal{L}_{L}^{2,\lambda}}:=\sup_{B\subset {\mathbb R}^n}\Big({1\over r_{B}^{\lambda}}
\int_B|(I-e^{-r_B^2 L})f(x)|^2dx\Big)^{1/2} <\infty.
\end{equation*}
\end{definition}

It is worth remarking that Definition \ref{def 3.1} is essentially equivalent to the original definition of a quadratic Campanato space  associated to $L$ introduced in \cite{DXY}. With this in mind, the following description of $\mathcal{L}_{L}^{2,\lambda}({\mathbb R}^{n})$ is quite natural.

\begin{lemma}\label{lemma3.2} Assume $L$ is a nonnegative self-adjoint operator on $L^2({\mathbb R^n})$ satisfying
\eqref{e2.1}. Let  $\lambda\in (0,n)$.  An element
$f\in{\mathcal{E}}_\lambda$ belongs to $\mathcal{L}_{L}^{2,\lambda}({\mathbb R}^{n})$ if and only if
\begin{eqnarray*}
\sup_{B\subset {\mathbb R}^n} \Big({1\over r_{B}^{\lambda}}
\int_B|(I-(I+r_B^2 L)^{-1})f(x)|^2dx\Big)^{1/2}<\infty.
\end{eqnarray*}
\end{lemma}
\begin{proof} This follows from a minor change of the argument for \cite[Lemma 8.1]{HM}.
\end{proof}

\subsection{Preduality for $\mathcal{L}_{L}^{2,\lambda}({\mathbb R}^{n})$}  By Theorem \ref{th2.10}, we  denote $F\dot{H}_{L,at,M}^{\lambda}({\mathbb R}^{n})$ by $F\dot{H}_{L,at}^{\lambda}({\mathbb R}^{n})$.  Theorem \ref{t11} is split into two parts: Theorem \ref{th3.3} and its converse Theorem \ref{th3.4} below.

\begin{theorem}\label{th3.3} Assume $L$ is a nonnegative self-adjoint operator on $L^2({\mathbb R^n})$ satisfying
\eqref{e2.1}. Let $\lambda\in (0,n)$ and $M\ge 1$.   Then  for any $f\in \mathcal{L}_{L}^{2,\lambda}({\mathbb R}^{n})$, the linear
functional
$$
\ell(g):=\langle f,g\rangle,
$$
which is initially defined on the dense subspace of $\MM $ comprising finite linear combinations of $(2,1,\lambda,\epsilon)$-molecules, $\epsilon>\frac{n-\lambda}{2}$ and where the pairing $\langle ,\rangle$ acts between $\MM $ and its dual, has a unique
bounded extension to $F\dot{H}_{L,at}^{\lambda}({\mathbb R}^{n})$ with
$$
\|\ell\|_{(F\dot{H}_{L,at}^{\lambda})^\ast}\leq C\|f\|_{\mathcal{L}_{L}^{2,\lambda}},
\quad\mbox{for some} \,\, C \,\mbox{ independent of }f.
$$
\end{theorem}

\begin{proof} Let us prove first that given $(2,1,\lambda,\epsilon)$-molecule $m$ with $\varepsilon>\frac{n-\lambda}{2}$ one has
\begin{eqnarray}\label{e3.5}
|\langle f, m\rangle|\leq C\|f\|_{\mathcal{L}_{L}^{2,\lambda}}\quad\forall\quad f\in \mathcal{L}_{L}^{2,\lambda}({\mathbb R}^{n}).
\end{eqnarray}
Note that if $f\in \mathcal{L}_{L}^{2,\lambda}({\mathbb R}^{n})$ then $f\in (\MM )^{*}$ and hence
$(I-(I+r_B^2L)^{-1})f\in L^2_{\rm loc}$. Thus, with $B$ denoting the ball associated with $m$,
we may write
$$
\langle f, m\rangle=\int_{{\mathbb R}^n}(I-(I+r_B^2L)^{-1})f(x){\overline{m(x)}}\,dx+\Big\langle (I+r_B^2L)^{-1}f, m\Big\rangle=:I_1+I_2.
$$
So, \eqref{e3.5} follows from controlling $I_1$ and $I_2$ from above.

For the term $I_1$, we apply Cauchy--Schwarz's inequality, the $L^2$-normalization of $m$ and (\ref{e2.9}) to obtain

\begin{align*}
|I_1|&\leq \sum_{j=0}^{\infty}\Big(\int_{U_j(B)}
|(I-(I+r_B^2L)^{-1})f(x)|^2dx\Big)^{1/2}
\Big(\int_{U_j(B)}|m(x)|^2dx\Big)^{1/2} \\
&\leq \sum_{j=0}^{\infty}2^{-j\epsilon}
(2^jr_{B})^{-\lambda/2}\Big(\int_{U_j(B)}|(I-(I+r_B^2L)^{-1})f(x)|^2dx\Big)^{1/2}.
\end{align*}
Upon covering $U_j(B)$ by approximate $2^{jn}$ balls of the radius $r_{B}$ and using $\epsilon>(n-\lambda)/2$,  we obtain
\begin{align*}
|I_1|
&\leq \sum_{j=0}^{\infty}2^{-j\epsilon} (2^jr_{B})^{-\lambda/2}2^{jn/2}r_{B}^{\lambda/2}\|f\|_{\mathcal{L}_{L}^{2,\lambda}}\nonumber\\
&= \sum_{j=0}^{\infty}2^{-j\epsilon} 2^{j(n/2-\lambda/2)}\|f\|_{\mathcal{L}_{L}^{2,\lambda}}\\
&\leq C\|f\|_{\mathcal{L}_{L}^{2,\lambda}}.
\end{align*}

For the term $I_2$, we use
$$
L(I+r_B^2L)^{-1}=r_B^{-2}\big(I-(I+r_B^2L)^{-1} \big)
$$
and the definition of a $(2,1,\lambda,\epsilon)$-molecule to derive
\begin{align*}
|I_2| &\leq C\Big|\int_{\mathbb R^n} (I-(I+r_B^2L)^{-1})f(x)
{\overline{(r_B^{2}L)^{-1}m(x)}}\,dx\Big|\\
&\leq C\sum_{j=0}^{\infty}\Big(\int_{U_j(B)}
|(I-(I+r_B^2L)^{-1})f(x)|^2dx\Big)^{1/2}
\Big(\int_{U_j(B)}\!\!\!|(r_B^2L)^{-1}m(x)|^2dx\Big)^{1/2}\\
&\leq C\sum_{j=0}^{\infty}2^{-j\epsilon}
(2^jr_{B})^{-\lambda/2}\Big(\int_{U_j(B)}|(I-(I+r_B^2L)^{-1})f(x)|^2dx\Big)^{1/2}.
\end{align*}
Furthermore, upon covering each $U_j(B)$ by approximate $2^{jn}$ balls of the radius $r_{B}$, we obtain
$$
|I_2|\leq\sum_{j=0}^{\infty}2^{-j\epsilon} (2^jr_{B})^{-\lambda/2}2^{jn/2}r_{B}^{\lambda/2}
\|f\|_{\mathcal{L}_{L}^{2,\lambda}}\leq C\|f\|_{\mathcal{L}_{L}^{2,\lambda}}.
$$

Our next goal is to show that for every
$N\in{\Bbb N}$ and for every $g=\sum_{j=0}^N \lambda_j a_j\in F\dot{H}_{L,at}^{\lambda}({\mathbb R}^{n})$, where
$\{a_j\}_{j=0}^N$ are $(2,2n,\lambda)$-atoms, we have

\begin{eqnarray}\label{e3.8}
\Big|\int_{\mathbb R^n} f(x){g(x)}\, dx\Big|
&\leq& C\big\|g\big\|_{F\dot{H}_{L,at}^{\lambda}}\big\|f\big\|_{\mathcal{L}_{L}^{2,\lambda}}.
\end{eqnarray}

Since the space of all finite linear combinations of $(2,2n,\lambda)$-atoms
is dense in $F\dot{H}_{L,at}^{\lambda}$, the linear functional $\ell$ will then have a unique
bounded extension to $F\dot{H}_{L,at}^{\lambda}$ defined in a standard fashion by continuity.
Below is a demonstration of (\ref{e3.8}). By Theorem~\ref{th2.11}, there is a
representation of
$$
g=\sum_{j=0}^N \lambda_j a_j=\sum_{i=0}^{K}\mu_i m_i,
$$
where $\{m_i\}_{i=0}^K$ are $(2,n, \lambda, n)$-molecules (of course, they are  $(2,1, \lambda, n)$-molecules) and
\begin{eqnarray*}
\sum\limits_{i=0}^{K}|\mu_i|\leq C\|g\|_{F\dot{H}_{L,at}^{\lambda}}.
\end{eqnarray*}
Therefore, by (\ref{e3.5}) we have
\begin{eqnarray*}
\Big|\int_{\mathbb R^n} f(x){g(x)}\,dx\Big|
&\leq& \sum\limits_{i=0}^{K}|\mu_i|
\Big|\int_{\mathbb R^n} f(x)m_i(x)\,dx\Big|\\
&\leq&C\sum\limits_{i=0}^{K}|\mu_i|\|f\|_{\mathcal{L}_{L}^{2,\lambda}}\\
&\leq&C\|g\|_{F\dot{H}_{L,at}^{\lambda}}\|f\|_{\mathcal{L}_{L}^{2,\lambda}},
\end{eqnarray*}
whence reaching (\ref{e3.8}) which in turn finishes the proof of
Theorem~\ref{th3.3}.
\end{proof}

Our next result is essentially the converse of Theorem~\ref{th3.3}.

\begin{theorem}\label{th3.4}  Assume $L$  is   a nonnegative self-adjoint operator   on $L^2({\mathbb R^n})$ satisfying \eqref{e2.1}.
Let  $\epsilon>0$ and $\lambda\in (0,n)$.  Suppose
  $\ell$ is a
bounded linear functional on $F\dot{H}_{L,at}^{\lambda}({\mathbb R}^{n})$. Then $\ell\in \mathcal{L}_{L}^{2,\lambda}({\mathbb R}^{n})$, and for
any $g\in F\dot{H}_{L,at}^{\lambda}({\mathbb R}^{n})$ (which can be represented as finite linear combinations of
$(2,1,\lambda,\epsilon)$-molecules) there holds
\begin{align}\label{ee4.1}
\ell(g)=\langle \ell,g\rangle,
\end{align}
where the pairing is that between $\MM$ and its dual. Moreover,
$$
\|\ell\|_{\mathcal{L}_{L}^{2,\lambda}}\leq C\|\ell\|_{(F\dot{H}_{L,at}^{\lambda})^\ast}.
$$
\end{theorem}

\begin{proof} By Theorem~\ref{th2.9},  we have that for any
$(2,1,\lambda,\epsilon)$-molecule $m$,
$$
\|m\|_{F\dot{H}_{L,at}^{\lambda}}\leq C\ \ \hbox{and\ so}\ \
|\ell(m)|\leq C\|\ell\|_{(F\dot{H}_{L,at}^{\lambda})^\ast}.
$$
By the discussion on Definition \ref{de41}, $\ell$ is
a bounded linear functional on $\mathcal{M}^{2,\lambda,\epsilon}(L)$ for any $\epsilon>0$. Thus,
$\ell\in{\mathcal{E}}_\lambda$ and (\ref{ee4.1}) holds.  Further, $(I-(I+r_B^2L)^{-1})\ell$ is well defined and belongs to
$L^2_{\rm loc}({\mathbb R}^n).$  Fix a ball $B$, and let $\varphi\in L^2(B)$, with
$\|\varphi\|_{L^2(B)}\leq 1$. As we observed before,
$$
{\widetilde m}:=r_{B}^{-\lambda/2}(I-(I+r_B^2L)^{-1})\varphi
$$
is (up to a multiplicative constant) a $(2,1,\lambda,\epsilon)$-molecule. Thus,

\begin{align*}
r_{B}^{-\lambda/2}|\langle(I-(I+r_B^2L)^{-1})\ell,\varphi\rangle|&=r_{B}^{-\lambda/2}|\langle\ell,(I-(I+r_B^2L)^{-1})\varphi\rangle|\\
&=|\langle \ell, {\widetilde m}\rangle|\\
&\leq C\|\ell\|_{(F\dot{H}_{L,at}^{\lambda})^\ast}.
\end{align*}

\noindent Taking the supremum over all such $\varphi$ supported in $B$, we
obtain
$$
 {1\over r_{B}^{\lambda}}\int_B|(I-(I+r_B^2L)^{-1})\ell(x)|^2dx \leq C
\|\ell\|^2_{(F\dot{H}_{L,at}^{\lambda})^\ast}.
$$
Finally, taking the supremum over all balls $B$ in ${\mathbb R}^n$, we arrive at
the conclusion of Theorem \ref{th3.4}.
\end{proof}

\begin{proof}[Proof of Theorem \ref{t11}] Combining Theorem \ref{th3.3}, Theorem \ref{th3.4} and Theorem \ref{th2.10}, we get
$$
({F}\dot{H}_{L}^{\lambda}({\mathbb R}^n))^*=({F}\dot{H}_{L,at}^{\lambda}({\mathbb R}^n))^*=\mathcal{L}_L^{2,\lambda}({\mathbb R}^n),
$$
thereby reaching the preduality stated in Theorem \ref{t11}.
\end{proof}

\end{document}